\documentclass[12pt]{article}
\usepackage{amssymb,amsbsy,amsmath,amsfonts,amssymb,amscd,amsthm}
\usepackage[T1]{fontenc}
\usepackage[utf8]{inputenc}
\usepackage{graphicx}
\usepackage[all]{xy}
\usepackage{pdfpages}
\usepackage[linktocpage]{hyperref}

\usepackage{latexsym, amssymb, amsthm}
\usepackage{dsfont}
\usepackage{color}
\usepackage{enumerate}

\newcommand{\RR}{\mathds{R}}
\newcommand{\CC}{\mathds{C}}

\newcommand{\LL}{{\mathcal L}}
\newcommand{\fra}{\mathfrak a}

\DeclareMathAlphabet\gothic{U}{euf}{m}{n}

\newtheorem{theorem}{Theorem}
\newtheorem{lemma}[theorem]{Lemma}
\newtheorem{proposition}[theorem]{Proposition}

\newtheorem{remark}[theorem]{Remark}

\numberwithin{theorem}{section}
\numberwithin{equation}{section}

\date{}
\begin{document}
\title{The square root  of a parabolic operator}
\author{El Maati Ouhabaz}

\maketitle     

\begin{abstract}
Let $L(t) = - {\rm div} \left( A(x,t) \nabla_x \right)$ for $t \in (0, \tau)$ be a uniformly elliptic operator with boundary conditions on a domain $\Omega$ of $\RR^d$ and $\partial = \frac{\partial}{\partial t}$. Define the parabolic operator $\LL = \partial + L$ on $L^2(0, \tau, L^2(\Omega))$ by  $(\LL u)(t) := \frac{\partial u(t)}{\partial t} + L(t)u(t)$. We assume a very little of regularity for the boundary of $\Omega$ and we assume that the coefficients $A(x,t)$ are measurable in $x$ and piecewise $C^\alpha$ in $t$ (uniformly in $x \in \Omega$) for some $\alpha > \frac{1}{2}$.  We prove the Kato square root property for  $\sqrt{\LL}$ and the estimate 
$$\| \sqrt{\LL}\, u \|_{L^2(0,\tau, L^2(\Omega))} \approx  \| \nabla_x u \|_{L^2(0,\tau, L^2(\Omega))}  + \| u \|_{H^{\frac{1}{2}}(0,\tau, L^2(\Omega))} + \left( \int_0^\tau \| u(t) \|_{L^2(\Omega)}^2\, \frac{dt}{t} \right)^{1/2}.$$ We also prove $L^p$-versions of this result. 
\end{abstract}

\vspace{5mm}
\noindent

\vspace{5mm}

\vspace{1cm}

\noindent
{\bf Home institution:}    \\[1mm]
Institut de Math\'ematiques de Bordeaux \\ 
Universit\'e de Bordeaux, CNRS, UMR 5251,  \\ 
351, Cours de la Lib\'eration.
33405 Talence, France.\\
Elmaati.Ouhabaz@math.u-bordeaux.fr\\
ORCID: 0000-0003-0849-3957. \\[8mm]


\section{Introduction and the main results}\label{sec0}

Consider on $L^2(\RR^d)$ the differential  operator $ L(t) = - {\rm div} \left( A(x,t) \nabla_x \right)$ where the matrix  $A(x,t) = (a_{kl}(x,t))_{1 \le k,l \le d}$ has complex measurable entries and satisfies the usual ellipticity condition 
\begin{equation}\label{eq0}
{\rm Re} \langle A(x,t) \xi, \xi \rangle \ge \kappa |\xi|^2,   \quad  | \langle A(x,t) \xi, \zeta \rangle | \le C |\xi| |\zeta |
\end{equation}
for all $\xi, \zeta \in \CC^d$, where $\kappa, C$ are positive constants independent of $(x,t) \in \RR^d\times \RR$ and $\langle ., . \rangle$ denotes the scalar product of $\CC^d$.
We consider the first order differential operator $\partial u = \frac{\partial u(t)}{\partial t}$ for all $u$ in the Sobolev space $H^1(\RR, L^2(\RR^d))$. One defines the half-order derivative $\partial^{1/2}$ by
\[ \partial^{1/2} u (t)  = - \frac{1}{2 \sqrt{2\pi}} \int_\RR \frac{1}{ | t-s|^{3/2}} ( u(t) -u(s) )\, ds.
\]
The following theorem is a parabolic version of the Kato square root property. It is proved by P. Auscher, M. Egert and K. Nystr\"om \cite{AEN}.
\begin{theorem}\label{thm-AEN}
Suppose \eqref{eq0}. There exists a realization of the  parabolic operator $\LL := \partial + L$ which is maximal accretive on $L^2(\RR^{d+1})$, the domain of its square root $\sqrt{\LL}$ coincides with 
$H^{\frac{1}{2}}(\RR, L^2(\RR^d)) \cap L^2(\RR, H^1(\RR^d))$ and 
\[ \| \sqrt{\LL} u \|_{L^2(\RR^{d+1})} \approx \| \nabla_x u \|_{L^2(\RR^{d+1})} + \| \partial^{1/2} u \|_{L^2(\RR^{d+1})}
\]
for all $u \in D(\sqrt{\LL})$.
\end{theorem}
A similar  result was proved by K. Nystr\"om \cite{Ny} in the case where $A(x,t) = A(x)$. 

The aim of the present short paper  is twofold. We consider  the above parabolic Kato  square root problem for operators on domains with boundary conditions.  Secondly,  we   investigate the problem on $L^p(0, \tau, L^r(\Omega))$ and not only on $L^2(0, \tau, L^2(\Omega))$. We  consider the time variable $t$ in an interval $(0, \tau)$ which is usual for evolution equations rather than the whole set $\RR$. In order to give the precise statements of our results we need some preparation. 

Let $\Omega$ be an open subset of $\RR^d$ with boundary $\Gamma$. Consider a closed subspace $V$ of $H^1(\Omega)$ which contains $H^1_0(\Omega)$ and define the sesquilinear form
\[
\fra(t,u,v) = \int_\Omega A(x,t) \nabla_x u. \overline{\nabla_x v}\, dx
\]
with domain $V$. We assume that the matrix $A(x,t) = (a_{kl}(x,t))_{1 \le k,l\le d}$ satisfies the ellipticity condition \eqref{eq0} with constants independent of $(x,t) \in \Omega \times (0, \tau)$. The associated operator is formally given by $ L(t) = - {\rm div} \left( A(x,t) \nabla_x \right)$ and subject to the boundary conditions fixed  by $V$. 
We say that $\fra$ is $C^\alpha$ for some $\alpha > 0$ if there exists a positive constant $M$ such that for all $u, v \in V$
\[ | \fra(t,u,v) - \fra(s,u,v) | \le M |t-s|^\alpha \| u \|_{V} \| v \|_{V}.
\]
We say that $\fra$ is piecewise $C^\alpha$ for some $\alpha > 0$ if there exist $\tau_1 = 0  < \tau_2 <  \cdots < \tau_N = \tau$ such that on each sub-interval $(\tau_j, \tau_{j+1})$, $\fra$ is the restriction of a $C^\alpha$ form on $[\tau_j, \tau_{j+1}]$. 

We make the following two assumptions. Suppose that $\fra$ is piecewise $C^\alpha$ for some $\alpha > \frac{1}{2}$. Observe  that this  is satisfied if the coefficients $a_{kl}, 1\le k,l \le d$,   are piecewise $C^\alpha$ in the $t$-variable, uniformly in the $x$-variable, for some $\alpha > \frac{1}{2}$. \\
Next, we assume that for each fixed $t$, the operator $L(t)$ satisfies the following Kato square root property
\begin{equation}\label{eq2}
V \subseteq D\left( \sqrt{L(t)} \right)  \quad {\rm and } \quad  \left\| \sqrt{L(t)}\,  f \right\|_{L^2(\Omega)}  \le C \left[ \left\| \nabla_x f \right \|_{L^2(\Omega)} + \left\| f \right \|_{L^2(\Omega)}  \right], \ f \in V.
\end{equation}
The constant $C$ is independent of $t$. By a well known  duality argument, \eqref{eq2}  implies that $D\left( \sqrt{L(t)} \right) = V $ and the norms
$\left\| \sqrt{L(t)}\,  f \right\|_{L^2(\Omega)} + \left\|  f \right\|_{L^2(\Omega)}$ and $\left\| \nabla_x f \right \|_{L^2(\Omega)} + \left\| f \right \|_{L^2(\Omega)}$ are equivalent. 
In many cases, the homogeneous  estimate 
\begin{equation}\label{eq2-0}
D\left( \sqrt{L(t)} \right) = V \quad {\rm and } \ \  \left\| \sqrt{L(t)}\,  f \right\|_{L^2(\Omega)} \approx \left\| \nabla_x f \right \|_{L^2(\Omega)}
\end{equation}
holds. The implicit constants in the equivalence $\approx$ are independent of $t$ since they depend only on the ellipticity constants. \\
The square root property  \eqref{eq2-0} is  always satisfied if $A(x,t)$ is symmetric. It is  satisfied if $\Omega = \RR^d$ by the solution of the Kato square root problem (see P. Auscher et al. \cite{AHLMT}). The non-homogeneous estimate \eqref{eq2} is satisfied if one has in addition  terms of lower order. On domains, \eqref{eq2} is satisfied if the boundary of $\Omega$ has a little of regularity (for example Lipschitz is enough) for Dirichlet boundary conditions ($V = H_0^1(\Omega)$),  Neumann  boundary conditions ($V = H^1(\Omega)$) or even for mixed boundary conditions. For this we refer to M. Egert, R. Haller-Dintelmann and P. Tolksdorf  \cite{Eg2}, the recent paper  
of S. Bechtel, M. Egert and R. Haller-Dintelmann \cite{BEHD} and the references therein.

\medskip
 Now we state our first main result. 

\begin{theorem}\label{thm1}
Suppose that the ellipticity condition \eqref{eq0} holds on $\Omega$. Suppose also  \eqref{eq2} and  that $\fra$ is  piecewise $C^\alpha$  for some $\alpha > \frac{1}{2}$.  Then there exists a realization of the parabolic operator $\LL := \partial + L$ that  is maximal accretive on $L^2(0, \tau, L^2(\Omega))$ and satisfies the Kato square root property
\[ D\left(\sqrt{\LL} \right) = \left\{ u \in H^{\frac{1}{2}}(0, \tau, L^2(\Omega)) \cap L^2(0, \tau, V), \ \int_0^\tau \| u(t) \|_{L^2(\Omega)}^2  \, \frac{dt}{t}< \infty \right\}
\]
and
\[
\left\| \sqrt{\LL}\,  u \right\|_{L^2(0,\tau, L^2(\Omega))} \approx  \left\| \nabla_x u \right\|_{L^2(0,\tau, L^2(\Omega))} + \| u \|_{H^{\frac{1}{2}}(0,\tau, L^2(\Omega))} + \left( \int_0^\tau \| u(t) \|_{L^2(\Omega)}^2\, \frac{dt}{t} \right)^{1/2}
\]
for all $u \in D\left(\sqrt{\LL} \right)$. 
\end{theorem}

Here and throughout this paper, $H^{\frac{1}{2}}(0, \tau, L^2(\Omega)) $ is  the usual  fractional Sobolev space of order $\frac{1}{2}$. It is defined as the complex  interpolation space $\left[ H^1(0, \tau, L^2(\Omega)), L^2(0, \tau, L^2(\Omega)) \right]_{\frac{1}{2}}$. Every function in 
$ H^{\frac{1}{2}}(0, \tau, L^2(\Omega)) $ is the restriction to $(0, \tau)$ of a function in $ H^{\frac{1}{2}}(\RR, L^2(\Omega))$.  The later space is  defined as usual by using  the Fourier transform. For all this we refer to J.L. Lions and E. Magenes \cite{LM}, Chapter 3, Section 5. 

\medskip

Our main idea for the proof of  the above  result is to make use of   the maximal regularity of the non-autonomous evolution equation 
\begin{equation}\label{eq:evol-eq} \tag{P}
\left\{
  \begin{array}{rcl}
     \frac{ \partial u(t)}{\partial t}  + L(t)\,u(t) &=& f(t), \ t \in (0, \tau] \\
     u(0)&=&u_0.
  \end{array}
\right.
\end{equation}
This maximal regularity was proved in an abstract setting by B. Haak and E.M. Ouhabaz \cite{HO} under the assumption that  the form $\fra$ is piecewise $C^\alpha$ for some $\alpha > \frac{1}{2}$.  It is also proved there that the maximal $L^p$-regularity holds if there exists a  non-decreasing function  $\omega: [0, \tau] \to [0, \infty)$  such that for $u, v  \in V$
\[  \left| \fra(t,u,v) - \fra(s,u,v) \right| \le \omega(|t{-}s|) \, \| u \|_{V} \| v \|_{V}
\]
with 
\begin{equation}  \label{eq1}
   \int_0^\tau \frac{\omega(t)}{t^{\frac{3}{2}}}\, dt < \infty. 
\end{equation}
 See also M. Achache and E.M. Ouhabaz \cite{AO} and the references there for an account on recent development on this topic.
 
 \medskip
  The  idea of using the maximal regularity in the proof of  Theorem \ref{thm1} lies in the fact  that we have a relatively precise description of the domain of the maximal accretive operator $\LL$.  Then, with the help of  imaginary powers (or a holomorphic functional calculus) of $\LL$ we can appeal to results on interpolation spaces which in turn give the description of $D\left(\sqrt{\LL} \right)$. Note that the proof of Theorem \ref{thm-AEN} in \cite{AEN} is very different and it is based on the first order approach initiated  by A. McIntosh and his collaborators (see e.g., A. Axelsson, S. Keith and A. McIntosh \cite{AKM}). One may wonder whether  the (piecewise) regularity in $t$ which we require in Theorem \ref{thm1} can be removed.  The first strategy to do this is to try to  adapt the proof in \cite{AEN} to parabolic operators on domains. This is not known and  seems to be a difficult problem. The second strategy is to prove the maximal regularity for \eqref{eq:evol-eq}  when the coefficients are merely bounded measurable in $t$ (and $x$).  This is a challenging open problem which was mentioned by J.L. Lions in 1961 and remains open. Note that an example of a family of forms $b(t,\cdot,\cdot)$  such that $t \mapsto b(t,u,v)$ is $C^{\frac{1}{2}}$ in $(0, \tau)$ but the corresponding family of operators does not have the maximal regularity is given by Fackler  \cite{Fackler}. Note however that these are not differential  operators. \\

  Our approach is quite flexible and applies without any additional effort to other situations such as operators with lower order terms, degenerate operators, systems and operators on weighted spaces. For clarity of exposition we do not search for generality and  we keep the setting described above.  Instead, we consider another problem which was not studied  before in the literature. We study the problem of the square root of $\LL$ on  $L^p(0, \tau, L^2(\Omega))$ for $p \not=2$.  For this we shall need the following  slightly stronger condition than \eqref{eq1}
\begin{equation}  \label{om-p}
    \int_0^\tau \frac{\omega(t)}{t^{1+\beta}}\, dt < \infty
\end{equation}
for some $\beta > \frac{1}{2}$. Clearly, \eqref{om-p} is satisfied if the coefficients $a_{kl}$ are $C^\alpha$ in $t$ (uniformly in $x$)  for some 
$\alpha > \frac{1}{2}$ .\\
We prove the following result. 
\begin{theorem}\label{thm1-p}
Suppose the ellipticity condition \eqref{eq0} on $\Omega$. Suppose  \eqref{eq2} and \eqref{om-p} and let $p \in (1, \infty)$ with $p \not=2$.  There exists a realization of the parabolic operator $\LL := \partial + L$ that  is maximal accretive on $L^p(0, \tau, L^2(\Omega))$ and satisfies the Kato square root property
\[
 \| \sqrt{\LL}\, u \|_{L^p(0,\tau, L^2(\Omega))} \approx \| u \|_{W^{\frac{1}{2},p}(0,\tau,L^2(\Omega))} + \| \nabla_x u \|_{L^p(0,\tau, L^2(\Omega))}.
\]
If $p \in (1, 2)$, the domain of $\sqrt{\LL}$ coincides with $W^{\frac{1}{2},p}(0, \tau, L^2(\Omega)) \cap L^p(0, \tau, V)$. If $p \in (2, \infty)$ this domain coincides with $\{u \in W^{\frac{1}{2},p}(0, \tau, L^2(\Omega)) \cap L^p(0, \tau, V), \ u(0) = 0 \}$.
\end{theorem}

 The ideas in the proof are  similar to the case of $p=2$ in the sense that we use the maximal regularity of \eqref{eq:evol-eq}  and estimates for imaginary powers 
$\partial^{is}$, $L^{is}$ and $(\nu + \LL)^{is}$ for some constant $\nu \ge 0$. While for $p=2$, the boundedness of  $\LL^{is}$ follows from the accretivity  of the operator $\LL$ on the Hilbert space $L^2(0, \tau, L^2(\Omega))$, the situation for $p \not=2$ requires some additional work. In order to prove the boundedness of $(\nu + \LL)^{is}$ we use a perturbation result for the holomorphic functional calculus due to J. Pr\"uss and G. Simonett \cite{PS}. The regularity condition \eqref{om-p} will be used both to ensure the maximal $L^p$-regularity and to prove a commutator estimate in order to apply the perturbation theorem in \cite{PS}. As a result, we prove that the maximal accretive operator $\nu + \LL$ has a bounded holomorphic functional calculus on $L^p(0, \tau, L^2(\Omega))$ for all $p \in (1,\infty)$. This latter result uses only the maximal regularity through the condition \eqref{om-p} and not  the square root property \eqref{eq2}.\\
Theorem \ref{thm1-p} shows that the Kato square root property for the parabolic operator $\LL$ holds beyond the Hilbert space setting $L^2(0, \tau, L^2(\Omega))$. A natural question arises whether one might prove a similar result on $L^p(0, \tau, L^r(\Omega))$  for some (or all) $r \not= 2$. We prove such a result for  time independent coefficients. The general case is more complicate and remains open unless the coefficients are smooth with respect to the space variable. See the last section of the paper. \\

Throughout the paper we use $\| . \|_E$ to denote the norm of a given Banach space $E$. All inessential constants are often denoted by $C, C'...$, the notation $ A \approx B$ means that there exists a constant  $C > 0$ such that $\frac{1}{C} A \le B \le C A$. \\

\noindent{\bf Acknowledgements.} The author would like to thank Sebastian Bechtel for several  interesting remarks and comments on an earlier version of this paper and 
 Moritz Egert and Sylvie Monniaux for stimulating discussions.  Thanks are due also to the reviewer for his/her comments on the paper. \\
This research is partly supported by the ANR project RAGE,  ANR-18-CE-0012-01.
\section{Proof of Theorem \ref{thm1}}
We start by recalling the following maximal regularity result from \cite{HO} (Theorem 2 and Corollary 4).  It is proved there in an abstract setting of time dependent forms having the same domain. We state it here for the case of elliptic operators as defined in the introduction, so we assume throughout this section that the ellipticity condition \eqref{eq0} is satisfied on $\Omega$. 
\begin{theorem}\label{thm-HO} 
  1) Suppose that $\fra$ is piecewise $C^\alpha$ for some $\alpha > \frac{1}{2}$ and that \eqref{eq2} holds.  Then the Cauchy problem \eqref{eq:evol-eq}  has
maximal $L^2$--regularity in $L^2(\Omega)$ for any given  $u_0  \in V$. In addition, there exists a positive constant $C$ such that 
\begin{equation}\label{apriori}
\| u \|_2 + \| \frac{\partial u}{\partial t} \|_2 + \| L(\cdot) u(\cdot) \|_2 \le C\left[  \| f \|_2 + \|u_0 \|_{V} \right].
\end{equation}
  2) Suppose \eqref{eq1}. 
   Then  \eqref{eq:evol-eq}, with $u_0 = 0$, has
maximal $L^p$--regularity in $L^2(\Omega)$ for all $p \in (1, \infty)$. If in addition $\omega$
satisfies the $p$--Dini condition
\begin{equation}  \label{eq:p-Dini}
     \int_0^\tau \left(\frac{\omega(t)}{t} \right)^p \,dt <  \infty,
\end{equation}
then  \eqref{eq:evol-eq} has maximal $L^p$--regularity for all $u_0 \in  (L^2(\Omega),
D(L(0)))_{1- \frac{1}{p}, p}$.\\
There exists a positive constant $C$ such that 
 \begin{equation}\label{apriori-p}
\| u \|_p + \| \frac{\partial u}{\partial t} \|_p + \| L(\cdot) u(\cdot) \|_p \le C\left[ \| f \|_p + \|u_0 \|_{(L^2(\Omega), D(L(0)))_{1-\frac{1}{p}, p}} \right].
\end{equation}
\end{theorem}

Recall that \eqref{eq:evol-eq} has maximal $L^p$--regularity in $L^2(\Omega)$ if for every $f \in L^p(0, \tau, L^2(\Omega))$ there exists a unique $u \in W^{1,p}(0, \tau, L^2(\Omega))$, $u(t) \in D(L(t))$ for a.e. $t \in (0, \tau)$ and $u$ satisfies \eqref{eq:evol-eq} for a.e. $t \in (0, \tau)$. 
We recall that $(L^2(\Omega), D(L(0)))_{1- \frac{1}{p}, p}$ is the real interpolation space and the $L^p$-norm in the apriori estimates \eqref{apriori} and \eqref{apriori-p} is the  norm of  $L^p(0, \tau, L^2(\Omega))$. 

Let us also mention that the maximal $L^2$--regularity holds under the slightly weaker  regularity property  that the map $t \mapsto L(t)$ is piecewise in $H^{\frac{1}{2}}(0, \tau, {\mathcal B}(V,V'))$ ($V'$ is the dual space of $V$) together with a minimal Dini condition. This is proved in \cite{AO} in an abstract setting.  As we mentioned in the introduction, it is not known whether the  maximal regularity holds for elliptic operators with measurable coefficients in the $t$-variable  (and in the $x$-variable as we do here).  The counter-example given in \cite{Fackler} is not a  differential operator.

We shall apply the previous theorem in the case where $u(0) = 0$. In this case, we have maximal $L^p$--regularity for every $p \in (1, \infty)$ provided $\fra$ satisfies \eqref{eq1}. If $\fra$ is discontinuous, we assume that it is piecewise $C^\alpha$ for some $\alpha > \frac{1}{2}$ and in addition \eqref{eq2} holds. For general forms, the condition $D\left(\sqrt{L(t)}\right) = V$  cannot be removed if $\fra$ has (at least) one jump,  see \cite{Di}. 

\medskip
Set ${\mathcal H} = L^2(0,\tau, L^2(\Omega))$ and define $\partial = \frac{\partial }{\partial t}$ with domain
\[
D(\partial) = {}_0H^1 := \{ u \in H^1(0, \tau, L^2(\Omega)), \, u(0) = 0 \}.
\]
Define also the operator $L$ by $(L u)(t) = L(t)u(t)$ with domain 
\[
D(L) = \left\{ u \in L^2(0, \tau, L^2(\Omega)),\, u(t) \in D(L(t)) \ {\rm a.e.}\  t \ {\rm and} \   L(\cdot)u(\cdot) \in {\mathcal H} \right\}.
\]
 
\begin{lemma}\label{lem1} Suppose either \eqref{eq1} or  $\fra$ is piecewise $C^\alpha$ for some $\alpha > \frac{1}{2}$ and \eqref{eq2} holds.  Define the parabolic operator 
\[ \LL = \partial + L \quad {\rm with\ domain} \ 
  \ D(\LL) =\!\!{}\ _0H^1 \cap D(L).
  \]
  Then $\LL$ is invertible, maximal accretive and has dense domain. The operators $\partial \LL^{-1}$ and $L \LL^{-1}$ are bounded on 
  ${\mathcal H}$.
 \end{lemma}
 \begin{proof} 
Integration by parts shows that $\partial$ is accretive. Then $\LL$ is  accretive as the sum of two accretive operators. It is invertible on $\mathcal H$ by Theorem \ref{thm-HO}. The fact that  $\partial \LL^{-1}$ and $L \LL^{-1}$  are  bounded operators on  ${\mathcal H}$ is a consequence of  the a priori estimate \eqref{apriori} (or \eqref{apriori-p}). 
A standard  duality argument shows that $\LL$ is densely defined. 
\end{proof}

Next, for a given $f \in {\mathcal H}$,  $u(t):= \int_0^t f(s)\, ds$ satisfies $ u \in D(\partial)$ and $\partial u = f$. Therefore $\partial$ is invertible and it is maximal accretive. In particular, this allows us to  define its square root $\sqrt{\partial}$ as a maximal accretive operator.  Similarly, $L$ is maximal accretive since  one checks that  $(( I + L)^{-1} u) (t) = (I + L(t))^{-1} u (t)$. Therefore, $\sqrt{L}$ is also well defined.

\begin{lemma}\label{lem2} Suppose either \eqref{eq1} or  $\fra$ is piecewise $C^\alpha$ for some $\alpha > \frac{1}{2}$ and \eqref{eq2} holds. We have 
\begin{equation}\label{eq5}
\| \sqrt{\partial}\, u \|_{\mathcal H} + \| \sqrt{L}\, u \|_{\mathcal H} \le C \| \sqrt{\LL}\, u \|_{\mathcal H}
\end{equation}
for all $u \in D(\sqrt{\LL})$. In particular, $D\left(\sqrt{\LL} \right) \subset D\left(\sqrt{\partial} \right) \cap D\left(\sqrt{L} \right)$. 
\end{lemma}

\begin{proof} 
 Since  $\partial$ and $L$  are  maximal accretive it is well  known (see e.g. \cite{Kato})  that  they have bounded imaginary powers 
\begin{equation}\label{eq3}
 \| \partial^{is}  \|_{{\mathcal B}({\mathcal H})} \le e^{\frac{\pi}{2} |s|} \quad {\rm and} \quad    \| L^{is}  \|_{{\mathcal B}({\mathcal H})} \le e^{\frac{\pi}{2} |s|},  \ s \in \RR.
 \end{equation}
For the same reason, $\LL$ also satisfies 
\begin{equation}\label{eq4}
 \| \LL^{is}  \|_{{\mathcal B}({\mathcal H})} \le e^{\frac{\pi}{2} |s|}, \ s \in \RR.
 \end{equation}
Define $T(z) := \partial^z \LL^{-z}$. Then for $z = is$ with $s \in \RR$, it follows from \eqref{eq3} and \eqref{eq4} that $T(is)$ is bounded on ${\mathcal H}$ with norm bounded by $e^{\pi |s|}$. Using Lemma \ref{lem1},  \eqref{eq3} and \eqref{eq4} we see  that $T(1+is)$ is also bounded on ${\mathcal H}$ with norm bounded by $C\, e^{\pi |s|}$. This implies  that  $\partial^{1/2} \LL^{-1/2}$ is a bounded operator on ${\mathcal H}$. Applying the same reasoning with $L$ in place of $\partial$ shows that $L^{1/2} \LL^{-1/2}$ is also bounded  on ${\mathcal H}$. This proves the lemma.
\end{proof}

\begin{lemma}\label{lem3} Suppose either \eqref{eq1} or   $\fra$ is piecewise $C^\alpha$ for some $\alpha > \frac{1}{2}$ and \eqref{eq2} holds. 
Then there exists a constant $c>0$ such that 
\begin{equation}\label{eq7}
 c \| \sqrt{\LL}\, u \|_{\mathcal H} \le \| \sqrt{\partial}\, u \|_{\mathcal H} + \| \sqrt{L}\, u \|_{\mathcal H}
\end{equation}
for all $u \in D(\sqrt{L}) \cap D(\sqrt{\partial})$. In particular, $ D\left(\sqrt{\partial} \right) \cap D\left(\sqrt{L} \right) \subset D\left(\sqrt{\LL} \right)$. 
\end{lemma}
\begin{proof} The  proof uses  a  duality argument.\\
 Firstly, one checks easily that the adjoint of $\partial$ is given by 
 \[ \partial^* v(t)  = - \frac{\partial v(t)}{\partial t}, \ D(\partial^*) = \{ v \in H^1(0, \tau, L^2(\Omega)),\  v(\tau) = 0 \}.
 \]
 The adjoint operator $L^*$ is defined similarly to  $L$ with $L(t)$ replaced by $L(t)^*$, i.e., $A(x,t)$ is replaced by its adjoint $A^*(x,t)$. On the other hand it is clear that the maximal regularity given by Theorem \ref{thm-HO} holds for the retrograde problem
 \begin{equation*}\label{eq:evol-eq-adjoint} 
\left\{
  \begin{array}{rcl}
     -\frac{ \partial v(t)}{\partial t}  + L(t)^*\,v(t) &=& f(t), \ t \in (0, \tau] \\
     v(\tau)&=&0.
  \end{array}
\right.
\end{equation*}
Using this we see as above that  the operator $\partial^* + L^*$, defined  on the intersection of the corresponding  domains, is invertible and it is maximal accretive. It turns out that this operator is the adjoint of $\LL$. Using the same proof as before,  Lemma \ref{lem2} applied to $\LL^*$ gives
 \begin{equation}\label{eq8}
\| \sqrt{\partial^*}\, v \|_{\mathcal H} + \| \sqrt{L^*}\, v \|_{\mathcal H} \le C \| \sqrt{\LL^*}\, v \|_{\mathcal H}
\end{equation}
for all $v \in D(\sqrt{\LL^*})$.\\
Let $ u \in D(\sqrt{L}) \cap D(\sqrt{\partial})$ and $v \in D(\sqrt{\LL^*})$. Then,
 \begin{eqnarray*}
 \left| ( u, \sqrt{\LL^*}\,  v )_{\mathcal H} \right| &=& \left| ( u, \LL^* (\LL^*)^{-1/2}\,  v )_{\mathcal H} \right| \\
 &=& \left| (  u, (\partial^* +L^*) (\LL^*)^{-1/2}\,  v  )_{\mathcal H} \right|\\
 &=& \left| (  \sqrt{\partial}\, u, \sqrt{\partial^*}\, (\LL^*)^{-1/2}\,  v  )_{\mathcal H}  + (  \sqrt{L}\, u, \sqrt{L^*}\, (\LL^*)^{-1/2}\,  v  )_{\mathcal H} \right|\\
 &\le& \left( \| \sqrt{\partial}\, u \|_{\mathcal H} + \| \sqrt{L}\, u \|_{\mathcal H} \right) \left( \| \sqrt{\partial^*}\,(\LL^*)^{-1/2}\, v \|_{\mathcal H} + \| \sqrt{L^*}\,(\LL^*)^{-1/2}\, v \|_{\mathcal H} \right)\\
 &\le& 2C \left( \| \sqrt{\partial}\, u \|_{\mathcal H} + \| \sqrt{L}\, u \|_{\mathcal H} \right) \left\|  v \right\|_{\mathcal H},
 \end{eqnarray*}
where we use  \eqref{eq8} to have the  final  inequality.  Hence, $v \mapsto ( u, \sqrt{\LL^*}\,  v )_{\mathcal H}$ extends to a continuous linear functional on $\mathcal H$. This implies that $u \in D\left(\sqrt{\LL} \right)$ as well as   \eqref{eq7}.
\end{proof}

\begin{proof}[Proof of Theorem \ref{thm1}]  

Under the sole assumption \eqref{eq1} or if  $\fra$ is piecewise $C^\alpha$ for some $\alpha > \frac{1}{2}$ and \eqref{eq2} holds we obtain from the previous lemmas that 
\begin{equation}\label{app-sqrt}
 \| \sqrt{\LL}\, u \|_{\mathcal H} \approx \| \sqrt{\partial}\, u \|_{\mathcal H} + \| \sqrt{L}\, u \|_{\mathcal H}
 \end{equation}
 for all $u \in D(\sqrt{\LL}) = D(\sqrt{L}) \cap D(\sqrt{\partial})$. 
On the other hand since  the operator $\partial$ has bounded imaginary powers it follows that $D(\sqrt{\partial})$ coincides with the complex interpolation space 
$[{}_0H^1, {\mathcal H} ]_{\frac{1}{2}}$. By  \cite{LM}, p. 68 or p. 257, this interpolation space coincides with
\[ \{ u \in H^{\frac{1}{2}}(0, \tau, L^2(\Omega)), \ \int_0^\tau \|u(t)\|_{L^2(\Omega)}^2\, \frac{dt}{t} < \infty \}.
\]
 In addition,  $\| \sqrt{\partial}\, u \|_{\mathcal H}$ is equivalent to  $\| u \|_{H^{\frac{1}{2}}(0,\tau, L^2(\Omega))} + \left( \int_0^\tau \| u(t) \|_{L^2(\Omega)}^2\, \frac{dt}{t} \right)^{1/2}$.\footnote{Remember that $\partial$ is invertible, hence the graph norm of $\sqrt{\partial} $ equivalent to $\| \sqrt{\partial}\,  u \|_{\mathcal H}$.}  As mentioned in the introduction, \eqref{eq2} implies that the quantities $\| \sqrt{L(t)}\, u(t) \|_{L^2(\Omega)} + \| u(t) \|_{L^2(\Omega)}$ and $\| \nabla_x u(t) \|_{L^2(\Omega)} + \| u(t) \|_{L^2(\Omega)}$ are equivalent with constants independent of $t \in (0, \tau)$. Therefore, $\| \sqrt{L}\, u \|_{\mathcal H} + \| u \|_{\mathcal H}$ and $\| \nabla_x u \|_{\mathcal H}  + 
 \| u \|_{\mathcal H} $ are equivalent. We use  this in \eqref{app-sqrt}  to obtain
\[
\| \sqrt{\LL}\, u \|_{\mathcal H} + \| u \|_{\mathcal H} \approx \| \sqrt{\partial}\, u \|_{\mathcal H} + \| \nabla_x u \|_{\mathcal H}  + 
 \| u \|_{\mathcal H}.
 \]
 From this and the fact that the operators  $\partial$ and  $\sqrt{\LL}$ are invertible  (cf. Lemma \ref{lem1}) we obtain the theorem. 
\end{proof}

\begin{remark}\label{rem}
1- In Theorem \ref{thm1} we could remove the (piecewise) regularity assumption in the $t$-variable by assuming that the Cauchy problem \eqref{eq:evol-eq} has maximal $L^2$-regularity in $L^2(\Omega)$. However, as we already mentioned in the introduction, it is not known whether  this maximal regularity is satisfied  when the coefficients $a_{kl}$ are merely measurable in $t$.\\ 
2-The proofs of Lemmas \ref{lem2} and \ref{lem3} do not use any specific property of the differential operators $L(t)$. These lemmas are valid in an abstract setting  of operators $L(t)$ which are associated with a family of sesquilinear forms 
\[ \fra : (0, \tau) \times V \times V \to \CC
\]
which are quasi-coercive and bounded with uniform constants in $t$. Here $V$ is a Hilbert space that is densely and continuously embedded into 
another given Hilbert space $H$. We define $\partial, L$ and $\LL$ as before. Under the sole assumption \eqref{eq1} we obtain $D(\sqrt{\LL}) = D(\sqrt{L}) \cap D(\sqrt{\partial})$ and 
 \[ \| \sqrt{\LL}\, u \|_{L^2(0,\tau, H)} \approx \| \sqrt{\partial}\, u \|_{L^2(0,\tau, H)} + \| \sqrt{L}\, u \|_{L^2(0,\tau, H)}.
 \]
 If $\fra$ is piecewise $C^\alpha$ for some $\alpha > \frac{1}{2}$, we assume in addition that \eqref{eq2} holds and we obtain the same conclusion. \\
 3- The ideas used in this section (as well as the next one) can also be  used to describe the domain of any  fractional power $D\left(\LL^\alpha \right)$ for 
 $\alpha \in (0,1)$.
\end{remark}

\section{$L^p(L^2)$-estimates}

In  the proofs of the previous section we used the maximal $L^2$-regularity given by Theorem \ref{thm-HO}. We take advantage that this latter theorem gives  also maximal $L^p$-regularity for every $p \in (1,\infty)$. We use this in the proof of  the $L^p(L^2)$-estimate of Theorem \ref{thm1-p}.\\
  Throughout this section we take the assumptions of Theorem \ref{thm1-p}, that is, we assume \eqref{eq0}, \eqref{eq2} and \eqref{om-p}. 
 
\medskip
 Fix $p \in (1, \infty)$ with $p \not=2$. Define on $L^p(0, \tau, L^2(\Omega))$ the operator $\partial = \frac{\partial}{\partial t}$  with domain 
\[
D(\partial) = {}_0W^{1,p} := \{ u \in W^{1,p}(0, \tau, L^2(\Omega)), \ u(0) = 0 \}.
\]
 It is well known  that $\partial$ has bounded imaginary powers on
$L^p(0, \tau, L^2(\Omega))$ (see e.g. \cite{DV}). It is not difficult to prove that $\partial$ is accretive and invertible. Hence, $\partial$ is maximal accretive. \\
As in the previous section, we define $L$ by $(L u)(t) := L(t)u(t)$ with domain
\[
D(L) = \left\{ u \in L^p(0, \tau, L^2(\Omega)),\, u(t) \in D(L(t)) \ {\rm a.e.}\  t \ {\rm and} \   L(\cdot)u(\cdot) \in L^p(0, \tau, L^2(\Omega)) \right\}.
\]
Then $L$ is maximal accretive. Since for fixed $t \in (0, \tau)$,
\[
\| L(t)^{is} \|_{{\mathcal B}(L^2(\Omega))} \le e^{\frac{\pi}{2} |s|}
\]
and  $(L^{is} u)(t) = L(t)^{is}u(t)$,\footnote{one starts from the resolvent formula $((\lambda I + L)^{-1}u)(t) = (\lambda I + L(t))^{-1}u(t)$ and then by integration along an appropriate contour to define the holomorphic functional calculus one obtains such a formula.} it follows that 
the operator $L$ has bounded imaginary powers on $L^p(0, \tau, L^2(\Omega))$. We define $\LL = \partial + L$ on the intersection of the domains. It follows from  Theorem \ref{thm-HO} that the operator $\LL$ is invertible. In particular, it is maximal accretive.  In contrast to the  Hilbert space setting  of Theorem \ref{thm1}, the boundedness of imaginary powers of $\LL$ is not a consequence of maximal accretivity. So we have to use a different  argument. 
\begin{proposition}\label{propPS}
There exists a $\nu \ge 0$ such that the operator $\LL + \nu$ has a bounded holomorphic functional calculus on $L^p(0, \tau, L^2(\Omega))$. In particular, $\LL + \nu$ has bounded imaginary powers.
\end{proposition}

The proof is based on the following perturbation theorem (see Corollary 3.2 in \cite{PS}). 
\begin{theorem}\label{thmps}
Let $A$ and $B$ be two operators having holomorphic functional calculi with angles $\phi_A$ and $\phi_B$ on a Banach space $X$. Suppose that $0 \in \rho(A)$, 
$B$ is ${\mathcal R}$-sectorial and  $\phi_A + \phi_B < \pi$. Suppose in addition that for some $0 \le \alpha < \beta < 1$ the Labbas-Terreni commutator estimate
\begin{equation}\label{ps}
\left\| A (\lambda + A)^{-1} \left[ A^{-1}(\mu + B)^{-1} - (\mu + B)^{-1}A^{-1} \right] \right\|_{{\mathcal B}(X)} \le C | \lambda|^{\alpha-1}
|\mu|^{-\beta-1}
\end{equation}
holds for all $\lambda$ and $\mu$ with $| \arg(\lambda) | < \pi - \phi_A$ and $| \arg(\mu) | < \pi- \phi_B$. Then there exists 
a $\nu \ge 0$ such that $\nu + A + B$ has a bounded holomorphic functional calculus on $X$.
\end{theorem}

\begin{proof} [Proof of Proposition \ref{propPS}] The operator $\LL + \epsilon $ is the sum of (non-commuting) operators $B = \partial$ and $ A = L + \epsilon$.  Each of these operators has a bounded holomorphic functional calculus on $L^p(0, \tau, L^2(\Omega))$ with angles $\phi_\partial = \frac{\pi}{2} + \epsilon'$ (for any $\epsilon' > 0$, see \cite{DV}) and $\phi_L < \frac{\pi}{2}$, respectively. Hence,  $\phi_\partial + \phi_L < \pi$. Next, the functional calculus is ${\mathcal R}-$bounded (for holomorphic functions with modulus $\le 1$). This follows from  \cite{Hy}, Theorem 10.3.4 (3)  in combination with  Proposition 7.5.3 (which shows that $L^p(0, \tau, L^2(\Omega)$ has Pisier's contraction principle since this is the case for the Hilbert space $L^2(\Omega)$). The role of $\epsilon > 0$ above is only to guarantee that $L+ \epsilon$ is invertible. For simplicity we forget $\epsilon$ and keep in mind that $L$ has to be replaced by $L + \epsilon$ in the sequel. We claim that \eqref{ps} is satisfied with $\alpha = \frac{1}{2}$ and $\beta$ as in \eqref{om-p}. Once this is proved we can apply Theorem \ref{thmps} to obtain the proposition. \\
Let $f \in L^p(0, \tau, L^2(\Omega))$ and set 
\[
I(t) := \left\| L(t) (\lambda + L(t))^{-1} \left[ L(t)^{-1}(\mu + \partial)^{-1} - (\mu + \partial)^{-1}L(t)^{-1}  f(t)\right] \right\|_{L^2(\Omega)}.
\]
Since 
\begin{equation}\label{res-truc}
(\mu + \partial)^{-1} f(t) = \int_0^t e^{-\mu (t-s)} f(s)\, ds
\end{equation} 
 and ${\rm Re} (\mu) \approx |\mu|$, we have
\begin{eqnarray*}
I(t) &=& \left\| \int_0^t e^{-\mu (t-s)} L(t) (\lambda + L(t))^{-1} \left[ L(t)^{-1} - L(s)^{-1} \right]  f(s)\, ds \right\|_{L^2(\Omega)}\\
&\le&  \int_0^t e^{- c |\mu|  (t-s)} \left\| L(t) (\lambda + L(t))^{-1} \left[ L(t)^{-1} - L(s)^{-1} \right] f(s) \right\|_{L^2(\Omega)}\, ds
\end{eqnarray*}
for some constant $c > 0$. Now we argue exactly as in \cite{OS}, p. 1675 to obtain
\[ \left\| L(t) (\lambda + L(t))^{-1} \left[ L(t)^{-1} - L(s)^{-1} \right] f(s) \right\|_{L^2(\Omega)} \le \frac{C}{| \lambda|^{1/2}} \omega(|t-s|) \| f(s) \|_{L^2(\Omega)}.
\]
This gives
\begin{equation}\label{tr}
 I(t) \le \frac{C}{| \lambda|^{1/2}} \int_0^t e^{- c |\mu|  (t-s)} \omega(|t-s|)  \| f(s) \|_{L^2(\Omega)} \, ds.
\end{equation}
The term $\int_0^t e^{- c |\mu|  (t-s)} \omega(|t-s|)  \| f(s) \|_{L^2(\Omega)} \, ds$  can be seen as an operator (acting on $\| f(s) \|_{L^2(\Omega)}$) with kernel 
$$K(t,s) = \chi_{(0,t)}(s) e^{- c |\mu|  (t-s)} \omega(t-s).$$
Using the assumption \eqref{om-p} we have for all $t \in (0,\tau)$
\begin{eqnarray*}
 \int_0^\tau K(t,s) \, ds &=& \frac{1}{|\mu|^{\beta+1}} \int_0^t e^{- c |\mu|  (t-s)} (| \mu |(t-s))^{\beta+1}\,  \frac{ \omega(t-s)}{(t-s)^{\beta+1}} \, ds\\
 &\le& \frac{C}{|\mu|^{\beta+1}} \int_0^\tau \frac{\omega(r)}{r^{1+\beta}} \, dr \le \frac{C'}{|\mu|^{\beta+1}}.
 \end{eqnarray*}
 Similarly, 
 \[
 \int_0^\tau K(t,s) \, dt \le   \frac{C'}{|\mu|^{\beta +1}},
 \]
 uniformly in $s \in (0, \tau)$. 
This implies that the operator with kernel $K(t,s)$ is bounded on $L^p(0,\tau)$ with norm bounded by $\frac{C'}{|\mu|^{\beta+1}}$. It follows from \eqref{tr} that the operator 
$L (\lambda + L)^{-1} \left[ L^{-1}(\mu + \partial)^{-1} - (\mu + \partial)^{-1}L^{-1} \right]$ is bounded on $L^p(0, \tau, L^2(\Omega))$ with norm bounded  by
$\frac{C}{|\lambda|^{1/2} |\mu|^{1+\beta}}$. This is exactly the condition \eqref{ps}.
\end{proof}

We go back to the proof of Theorem \ref{thm1-p}. Since $\partial$ has imaginary powers, we have  $D(\sqrt{\partial}) = [{}_0W^{1,p}, L^p(0, \tau, L^2(\Omega))]_{\frac{1}{2}}$ with equivalent norms. It follows from \cite{Amann}, Theorem 4.7.1 or \cite{Denk}, p. 41 that 
$[{}_0W^{1,p}, L^p(0, \tau, L^2(\Omega))]_{\frac{1}{2}}$ coincides with $W^{\frac{1}{2},p}(0, \tau, L^2(\Omega))$ if $p < 2$ and with
${}_0W^{\frac{1}{2},p}(0, \tau, L^2(\Omega))$ if $p > 2$.\footnote{This is stated in \cite{Amann} and \cite{Denk} on the interval $(0,\infty)$ instead of $(0, \tau)$. One either uses a similar retraction and coretraction argument used their to deal directly  with $(0,\tau)$ or use a cut-off argument around the point $\tau$. See also \cite{BEg} for interpolation results in the scalar case.}  Hence 
\begin{equation}\label{eq2-2}
\| \sqrt{\partial}\, u \|_{L^p(0,\tau, L^2(\Omega))} \approx \| u \|_{W^{\frac{1}{2},p}(0,\tau,L^2(\Omega))}.
\end{equation}
By Proposition \ref{propPS}, $(\LL + \nu)^{is}$ is bounded on $L^p(0,\tau, L^2(\Omega))$, thus  we can repeat  the proof of  Lemma \ref{lem2} and obtain
\[
 \| \sqrt{\partial}\, u \|_{L^p(0,\tau, L^2(\Omega))} + \| \sqrt{L}\, u \|_{L^p(0,\tau, L^2(\Omega))} \le C  \| \sqrt{\LL + \nu}\, u \|_{L^p(0,\tau, L^2(\Omega))}.
 \]
On the other hand since the operator $\LL$ is invertible by Theorem \ref{thm-HO}, we can remove the constant $\nu$ in the previous inequality and obtain
\[
 \| \sqrt{\partial}\, u \|_{L^p(0,\tau, L^2(\Omega))} + \| \sqrt{L}\, u \|_{L^p(0,\tau, L^2(\Omega))} \le C'  \| \sqrt{\LL }\, u \|_{L^p(0,\tau, L^2(\Omega))}.
 \]
Using the same estimate for the adjoint operator on $L^{p'}(0, \tau, L^2(\Omega))$ we argue by  duality as  in Lemma \ref{lem3} and obtain  
the reverse inequality. Therefore,
\begin{equation}\label{eq2-4}
 \| \sqrt{\LL}\, u \|_{L^p(0,\tau, L^2(\Omega))} \approx \| \sqrt{\partial}\, u \|_{L^p(0,\tau, L^2(\Omega))} + \| \sqrt{L}\, u \|_{L^p(0,\tau, L^2(\Omega))}
 \end{equation}
 for all $u \in D(\sqrt{\LL}) = D(\sqrt{\partial}) \cap D(\sqrt{L})$. 
Using \eqref{eq2-2} it follows that 
\begin{equation}\label{eq2-5}
 \| \sqrt{\LL}\, u \|_{L^p(0,\tau, L^2(\Omega))} \approx \| u \|_{W^{\frac{1}{2},p}(0,\tau,L^2(\Omega))} + \| \sqrt{L}\, u \|_{L^p(0,\tau, L^2(\Omega))}
 \end{equation}
 for all $u \in D(\sqrt{\LL}) = [{}_0W^{1,p}(0,\tau, L^2(\Omega)), L^p(0,\tau, L^2(\Omega))]_{\frac{1}{2}} \cap D(\sqrt{L})$. \\
Thus we have proved Theorem \ref{thm1-p}.\\

As we already mentioned before, the method we employed in this paper can be used in other circumstances. For example, the above $L^p(L^2)$-estimate can be proved for elliptic operators with lower order terms, some degenerate operators as well as parabolic systems. We do not write the details since they are essentially a simple  repetition of what is presented above. 

\section{$L^p(L^r)$-estimates}
In this section we address the question whether the previous results can be extended to $L^p(0,\tau, L^r(\Omega))$ for $r \not= 2$. When reproducing  the arguments of the previous sections we face two problems. The first is to have maximal $L^p$-regularity  in $L^r(\Omega)$ since Theorem \ref{thm-HO} is specific to the $L^2(\Omega)$ case. The second one is to have boundedness of imaginary powers of $\LL$ (or $\nu + \LL$ for some constant $\nu \ge 0$). The arguments in the proof of Proposition  \ref{propPS} use the sesquilinear form setting in order to check \eqref{ps}. Note that there are results on maximal regularity outside the Hilbert space (and hence the sesquilinear form) setting. However these results assume the  domains of $L(t)$ to be constant. See \cite{ACFP} and the references there. In order to guarantee that the operators $L(t)$ have the same domain on $L^r(\Omega)$ the natural thing to do is to compute this domain and show that it coincides with some Sobolev space. In order to do so one needs some regularity in the $x$-variable for $A(x,t)$  and also some regularity of $\Omega$. In order to stay with non-smooth coefficients in the $x$-variable we shall concentrate on the case $A(x,t) = A(x)$. We also assume that our elliptic operator is subject to the Dirichlet boundary conditions. With the same notation as before, we have
\begin{theorem}\label{thm1-p-r} Suppose \eqref{eq0} on $\Omega$, \eqref{eq2} and \eqref{om-p}. 
Suppose that $A(x,t) = A(x)$ has real-valued coefficients. Let $p \in (1, \infty)$ and denote by $p'$ its conjugate. Then for $r \in [\min(p,p'), \max(p,p')]$, 
\begin{equation}\label{eq3-1}
\| \sqrt{\partial}\, u \|_{L^p(0,\tau, L^r(\Omega))} + \| \sqrt{L}\, u \|_{L^p(0,\tau, L^r(\Omega))} \approx  \| \sqrt{\LL}\, u \|_{L^p(0,\tau, L^r(\Omega))}
 \end{equation}
for all $u \in D(\sqrt{\LL})$. In addition, for $r \in [\min(p,p'), 2]$, there exists a constant $C$ such 
\begin{equation}\label{eq3-2}
\| \sqrt{\partial}\, u \|_{L^p(0,\tau, L^r(\Omega))} + \| \nabla u \|_{L^p(0,\tau, L^r(\Omega))} \le C  \| \sqrt{\LL}\, u \|_{L^p(0,\tau, L^r(\Omega))}.
 \end{equation}
\end{theorem}
\begin{proof} Firstly,  since $L$ has real-coefficients and is subject to the Dirichlet boundary conditions, the semigroup $e^{-tL}$ is sub-Markovian (cf. \cite{Ouh}, Chapter 4). Therefore, by \cite{Lamberton}, $L$ has maximal $L^p$-regularity on $L^r(\Omega)$ for all $p, r \in (1, \infty)$. In particular, the operator $\LL = \partial + L$ defined on the intersection  ${}_0W^{1,2}(0, \tau, L^r(\Omega)) \cap D(L)$ is maximal accretive (note that both $\partial$ and $L$ are accretive on $L^p(0,\tau, L^r(\Omega))$). On the other hand, the two maximal accretive operators $\partial$ and $L$ are generators of positive semigroups. For positivity of $e^{-tL}$ see \cite{Ouh}, Chapter 4 and for $e^{-t\partial}$ this follows readily from the positivity of its resolvent (see \eqref{res-truc}). This and the Trotter product formula give the positivity of the contraction semigroup $e^{-t \LL}$ on  $L^p(0,\tau, L^r(\Omega))$. Since for  $r = p$,
$L^p(0,\tau, L^p(\Omega)) \simeq L^p(\Omega\times (0, \tau))$ we may use the transference method \cite{coif-weis} to obtain that 
$\LL$ has a bounded holomorphic functional calculus on $L^p(0,\tau, L^p(\Omega))$ (with angle $\phi > \frac{\pi}{2}$). This is also true for $\nu + \LL$ for any $\nu \ge 0$. Using this and Proposition \ref{propPS} it follows by interpolation that $\nu + \LL$ has a bounded holomorphic functional calculus on $L^p(0, \tau, L^r(\Omega))$ for 
$r \in [p,2]$ or $[2,p]$. What we did here for $\LL$ is also valid for $\LL^* = \partial^* + L^*$ by the same arguments. This gives that $\nu + \LL$ has a bounded holomorphic functional calculus on $L^p(0, \tau, L^r(\Omega))$ for all $p \in (1, \infty)$ and $r \in [\min(p,p'), \max(p,p')]$. In particular, the imaginary powers $(\nu + \LL)^{is}$ are bounded on these spaces. The rest of the proof of \eqref{eq3-1} is exactly the same as for Theorem \ref{thm1-p}.

Suppose now that $r \in [\min(p,p'), 2]$. Then the Riesz transform $\nabla L^{-\frac{1}{2}}$ is bounded on $L^r(\Omega)$ (see \cite{Ouh}, Section 7.7). This gives $\| \nabla f \|_{L^r(\Omega)} \le C \| \sqrt{L}\, f \|_{L^r(\Omega)}$. Thus, \eqref{eq3-2} follows from \eqref{eq3-1}.
\end{proof}

\begin{remark}\label{rem2}
The idea of using the  transference method on $L^p(0, \tau, L^p(\Omega))$ was already  used in \cite{OS2} in the context of parabolic Schr\"odinger operators.
\end{remark}

\end{document}